\documentclass[12pt,a4paper]{amsart}
\usepackage{amsfonts,amssymb,amsmath,amsthm}
\usepackage[utf8]{inputenc}
\usepackage[english]{babel}
\usepackage[left=2.5cm, top=2.5cm,bottom=2.5cm,right=2.5cm]{geometry}
\usepackage{graphicx}
\usepackage{tikz}
\usetikzlibrary{arrows.meta}
\usepackage{pgfplots}
\newtheorem{theorem}{Theorem}
\newtheorem*{theorem*}{Theorem}

\newtheorem{lem}[theorem]{Lemma}
\newtheorem{prop}[theorem]{Proposition}

\theoremstyle{definition}
\newtheorem{definition}[theorem]{Definition}

\numberwithin{equation}{section}
\newtheorem*{acknowledgement}{Acknowledgement}

\DeclareMathOperator{\var}{Var}

\newcommand{\N}{\ensuremath{{\mathbb N}}}

\newcommand{\R}{\ensuremath{{\mathbb R}}}

\newcommand{\eps}{\varepsilon}
\newcommand{\I}{\mathcal{I}}

\newcommand{\dint}{\,{\rm d}}

\author{ Simon Breneis}
\date{\today}
\address[Simon Breneis]{Institut f\"ur Analysis\\
Johannes Kepler Universit\"at Linz\\
Altenbergerstrasse 69\\
4040 Linz\\
Austria}
\email{simon.breneis@jku.at}
\thanks{The author is supported by the Austrian Science Fund (FWF) Project F5513-N26, which is a part of the Special Research Program ``Quasi-Monte Carlo Methods: Theory and Applications''.}

\keywords{}
\subjclass{}
\begin{document}

\title{On variation functions and their moduli of continuity}

\begin{abstract}
We study the moduli of continuity of functions of bounded variation and of their variation functions. It is easy to see that the modulus of continuity of a function of bounded variation is always smaller or equal to the modulus of continuity of its variation function. We show that we cannot make any reasonable conclusion on the modulus of continuity of the variation function if we only know the modulus of continuity of the parent function itself. In particular, given two moduli of continuity, the first being weaker than Lipschitz continuity, we show that there exists a function of bounded variation with minimal modulus of continuity less than the first modulus of continuity, but with a variation function with minimal modulus of continuity greater than the second modulus of continuity. In particular, this negatively resolves the open problem whether the variation function of an $\alpha$-Hölder continuous function is $\alpha$-Hölder continuous.
\end{abstract}

\maketitle

%\tableofcontents

% % % % % % % % % % % % % % % % % % % % % %
\section{Preliminaries and Motivation}
% % % % % % % % % % % % % % % % % % % % % % 

We define the variation of a function $f:[a,b]\to\R$ by $$\var(f;a,b):=\sup\bigg\{\sum_{i=1}^n\big|f(x_i)-f(x_{i-1})\big|\colon a=x_0\le x_1\le\dots\le x_n = b\text{ for some } n\in\N\bigg\}.$$ If the interval $[a,b]$ is clear from the context, we also write $\var(f):=\var(f;a,b)$. If the variation $\var(f)$ is finite, we say that $f$ is of bounded variation. Functions of bounded variation were first introduced by Jordan in \cite{J1881} in the study of Fourier series. By now, they have many applications, for example in the study of Riemann-Stieltjes integrals. We also refer the interested reader to the rather recent and comprehensive book on functions of bounded variation by Appell, Bana\'{s} and Merentes \cite{A2014}.

To every function $f$ of bounded variation, we can associate its variation function $\var_f:[a,b]\to\R$ defined by $$\var_f(x):=\var(f;a,x).$$ We also say that $f$ is the parent function of $\var_f$. Functions of bounded variation and their variation functions share many regularity properties. We state the following two connections that are most relevant to our studies, for further results, we refer the reader to \cite[78]{A2014} and \cite{H1976,R1979}.

\begin{prop}\label{prop:VariationFunctionParentFunctionConnections}
Let $f:[a,b]\to\R$ be a function of bounded variation.
\begin{enumerate}
\item The function $f$ is continuous if and only if its variation function $\var_f$ is continuous.
\item The function $f$ is Lipschitz continuous if and only if its variation function $\var_f$ is Lipschitz continuous.
\end{enumerate}
\end{prop}

For the convenience of the reader, we state a proof in Section \ref{sec2}. In \cite[78]{A2014}, the authors also proved that if the variation function $\var_f$ is $\alpha$-Hölder continuous, then also the parent function $f$ is $\alpha$-Hölder continuous. They noted that the inverse implication is still an open problem, see \cite[Problem 1.1 on p.100]{A2014}. We show that the inverse implication does not hold. A counterexample can be found in Section \ref{sec4}. In fact, we also show a more general statement involving moduli of continuity in Section \ref{sec3}, see Theorem \ref{thm:TheTheorem} below. Although Theorem \ref{thm:TheTheorem} implies the existence of a counterexample as in Section \ref{sec4}, we think that it is worth to work out this example in detail since it clearly shows the main ideas of the proof of the more general Theorem \ref{thm:TheTheorem}.

\begin{definition}
A continuous, increasing function $\omega:[0,\infty)\to[0,\infty)$ with $\omega(0)=0$ is called modulus of continuity.
\end{definition}

We remark that this is not the most general definition used for moduli of continuity. Often, the requirement that $\omega$ is increasing is dropped and the continuity is replaced with continuity at zero. The reason for our more restrictive definition is to achieve simpler and clearer statements and better consistency with the coming definitions. Lemma \ref{lem:MinimalModulusOfContinuityProperties} illustrates, however, that our definition is in some sense the most general one.

Moduli of continuity are usually not used by themselves. Instead, they are helpful in characterizing how continuous a given function is.

\begin{definition}
Let $\I\subset\R$ be a bounded or unbounded interval and let $f:\I\to\R$ be a function. A modulus of continuity $\omega$ is called a modulus of continuity for $f$ if for all $x,y\in \I$, we have $$\big|f(x)-f(y)\big|\le\omega\big(|x-y|\big).$$
\end{definition}

Examples of moduli of continuity are $x\mapsto Lx$ and $x\mapsto Lx^{\alpha}$ for $0<\alpha< 1$. They characterize the Lipschitz and the $\alpha$-Hölder continuous functions with Lipschitz and $\alpha$-Hölder constant $L$, respectively.

It is easy to see that given a function $f$ and two moduli of continuity $\omega_1\le\omega_2$, if $\omega_1$ is a modulus of continuity for $f$, so is $\omega_2$. In that sense, larger moduli of continuity represent weaker continuity conditions. In particular, to every continuous function we can associate its minimal modulus of continuity.

\begin{definition}
Let $\I\subset\R$ be a bounded or unbounded interval and let $f:\I\to\R$ be a continuous function. Then the minimal modulus of continuity of $f$ is defined as $$\omega_f(h):=\sup\Big\{\big|f(x)-f(y)\big|\colon x,y\in\I,\ |x-y|\le h\Big\}.$$
\end{definition}

We can now state the first, rather simple connection between functions of bounded variation, their variation functions and their moduli of continuity. The proof of this Proposition is in Section \ref{sec3}.

\begin{prop}\label{prop:VariationFunctionParentFunctionModuliOfContinuity}
Let $f:[a,b]\to\R$ be a continuous function of bounded variation. Then $\omega_f\le\omega_{\var_f}$.
\end{prop}

In light of Proposition \ref{prop:VariationFunctionParentFunctionConnections} and Proposition \ref{prop:VariationFunctionParentFunctionModuliOfContinuity}, one might hope that it is also possible to prove that $\omega_{\var_f}$ can be bound in terms of $\omega_f$. We show that this is not the case.

\begin{theorem}\label{thm:TheTheorem}
Let $\omega,\omega'$ be two moduli of continuity such that $$\lim_{h\to 0}\frac{\omega(h)}{h} = \infty,$$ and let $\omega'$ be bounded. Then there exists a function $f:[0,1]\to\R$ of bounded variation such that $\omega_f\le\omega$ and $\omega_{\var_f}\ge\omega'$.
\end{theorem}

The conditions in the theorem are as general as possible. The condition on $\omega$ is necessary, as otherwise $f$ is Lipschitz continuous, implying that also $\var_f$ is Lipschitz continuous by Proposition \ref{prop:VariationFunctionParentFunctionConnections}. Furthermore, since $f$ is of bounded variation, $\var_f$ is bounded. Hence, also the condition on $\omega'$ is necessary. The proof of this theorem is given in Section \ref{sec3}.

% % % % % % % % % % % % % % % % % % % % % %
\section{Proof of Proposition \ref{prop:VariationFunctionParentFunctionConnections}} \label{sec2}
% % % % % % % % % % % % % % % % % % % % % % 

First, by a simple application of the triangle inequality, we get that finer partitions capture more of the variation of a function in the following sense.

\begin{lem}\label{lem:VariationOnFinerLaddersHigher1D}
Let $f:[a,b]\to\R$ be a function and let $$a=x_0\le x_1\le\dots\le x_n=b$$ and $$a=y_0\le y_1\le\dots\le y_m=b$$ be two partitions of the interval $[a,b]$ with $$\big\{x_0,x_1,\dots,x_n\big\}\subset\big\{y_0,y_1,\dots,y_n\big\}.$$ Then $$\sum_{i=1}^n\big|f(x_i)-f(x_{i-1})\big|\le\sum_{i=1}^m\big|f(y_i)-f(y_{i-1})\big|.$$
\end{lem}

Furthermore, it is an easy exercise to show the following Lemma.

\begin{lem}\label{lem:VariationIsSummable}
Let $f:[a,b]\to\R$ be a function and let $c\in[a,b]$. Then $$\var(f;a,b)=\var(f;a,c)+\var(f;c,b).$$
\end{lem}

We now give a simple proof of Proposition \ref{prop:VariationFunctionParentFunctionConnections}. This proof is basically the same as the proof in \cite[78]{A2014}.

\begin{proof}[Proof of Proposition \ref{prop:VariationFunctionParentFunctionConnections}]
1. We prove that $f$ is right-continuous if and only if $\var_f$ is right-continuous. Similarly, $f$ is left-continuous if and only if $\var_f$ is left-continuous. Together, this shows that $f$ is continuous if and only if $\var_f$ is continuous.

Let $f$ be right-continuous at $x$. Let $\eps>0$ be arbitrary and let $\delta>0$ be such that $\big|f(x) - f(x+h)\big| < \eps/2$ for all $0\le h<\delta$. Let $x=x_0\le x_1\le \dots \le x_n=b$ be a partition of $[x,b]$ such that $$\sum_{i=1}^n|f(x_i)-f(x_{i-1})|\ge \var(f;x,b) - \eps/2.$$ Using Lemma \ref{lem:VariationOnFinerLaddersHigher1D}, we can assume that $x < x_1 < x+\delta$. Hence, we have
\begin{align*}
\var(f;x,b) &\le \sum_{i=1}^n\big|f(x_i)-f(x_{i-1})\big| + \eps/2 \le \big|f(x_{i_0}) - f(x)\big| + \sum_{i=2}^n\big|f(x_i)-f(x_{i-1})\big| + \eps/2\\
&\le \eps/2 + \var(f;x_1,b) + \eps/2 = \var(f;x_1,b) + \eps.
\end{align*}
Lemma \ref{lem:VariationIsSummable} implies that 
\begin{align*}
0 &\le \var_f(x_1)-\var_f(x) = \var(f;a,x_1) - \var(f;a,x) = \var(f;x,x_1)\\
&= \var(f;x,b) - \var(f;x_1,b) \le \eps
\end{align*}
holds for all $x_1\in (x,x+h)$. Thus, $\var_f$ is right-continuous at $x$.

On the other hand, if $\var_f$ is right-continuous at $x$, then it follows from Lemma \ref{lem:VariationIsSummable} that $$\big|f(x+h)-f(x)\big| \le \var(f;x,x+h) = \var(f;a,x+h) - \var(f;a,x) = \var_f(x+h)-\var_f(x),$$ which implies that $f$ is right-continuous at $x$.

2. If $f$ is Lipschitz continuous with constant $L$ and $a\le x\le y\le b$, then it is easily verified that $$\big|\var_f(x)-\var_f(y)\big| = \var(f;x,y) \le L|x-y|.$$ Conversely, if $\var_f$ is Lipschitz continuous with constant $L$ and $a\le x\le y\le b$, then $$\big|f(y)-f(x)\big| \le \var(f;x,y) = \var_f(y)-\var_f(x) \le L|y-x|.$$
\end{proof}

% % % % % % % % % % % % % % % % % % % % % %
\section{An Example} \label{sec4}
% % % % % % % % % % % % % % % % % % % % % % 

Before giving an example of a function that is $\alpha$-Hölder continuous with a variation function that is not $\gamma$-Hölder continuous for all $0<\gamma <1$, we need to generalize Lemma \ref{lem:VariationIsSummable} slightly. We denote by $f(x+)$ the right-sided limit of the function $f$ at $x$.

\begin{lem}\label{lem:TotalVariationSigmaSummable}
Let $f:[a,b]\to\R$ be a function with $f(a+)=f(a)$. Let $b=z_0>z_1>z_2>\dots$ be a strictly decreasing sequence in $[a,b]$ that converges to $a$. Then $$\var(f;a,b)=\sum_{n=0}^\infty\var(f;z_{n+1},z_n).$$
\end{lem}

\begin{proof}
First, the series converges (potentially to infinity), since all the terms are non-negative. Applying Lemma \ref{lem:VariationIsSummable}, we have for $k\in\N$ that $$\var(f;a,b)\ge \var(f;z_k,z_0) = \sum_{n=0}^{k-1}\var(f;z_{n+1},z_n).$$ Taking $k$ to infinity yields $$\var(f;a,b)\ge\sum_{n=0}^\infty\var(f;z_{n+1},z_n).$$

On the other hand, let $\eps>0$ and let $a=x_0< x_1\le\dots\le x_n=b$ be a partition of $[a,b]$. Let $k\in\N$ be such that $x_0<z_k<x_1$ and $\big|f(a)-f(z_k)\big|<\eps$. Such a $k$ exists, since $z_k\to a$ and therefore, $f(z_k)\to f(a)$. Lemma \ref{lem:VariationIsSummable} yields 
\begin{align*}
\sum_{i=1}^n\big|f(x_i)-f(x_{i-1})\big| &= \sum_{i=2}^n\big|f(x_i)-f(x_{i-1})\big| + \big|f(x_1)-f(x_0)\big|\\
&\le \var(f;x_1,b)+\big|f(x_1)-f(z_k)\big| + \big|f(z_k)-f(x_0)\big|\\
&\le \var(f;x_1,b) + \var(f;z_k,x_1) + \eps = \var(f;z_k,b) + \eps\\
&= \sum_{n=0}^{k-1}\var(f;z_{n+1},z_n) + \eps \le \sum_{n=0}^\infty\var(f;z_{n+1},z_n) + \eps.
\end{align*}
Since $\eps>0$ was arbitrary, $$\sum_{i=1}^n\big|f(x_i)-f(x_{i-1})\big| \le \sum_{n=0}^\infty\var(f;z_{n+1},z_n).$$ Taking the supremum over all partitions of $[a,b]$ yields $$\var(f;a,b) \le \sum_{n=0}^\infty\var(f;z_{n+1},z_n),$$ which proves the lemma.
\end{proof}

Let $0<\alpha<1$. We construct a function $f$ that is of bounded variation and $\alpha$-Hölder continuous, such that $\var_f$ is $\gamma$-Hölder continuous for no $\gamma\in(0,1)$.

First, consider the following general example. Let $x_1>x_2>\dots >0$ be a sequence with $x_n\to 0$ and let $(y_n)$ be a sequence with $y_2>y_4>\dots>0$, $y_{2n-1}=0$ for $n\in\N$ and $y_n\to 0$. Define the function $f:[0,x_1]\to\R$ as $f(x_n)=y_n$ and interpolate linearly in between. An example of such a function is shown in the picture below.

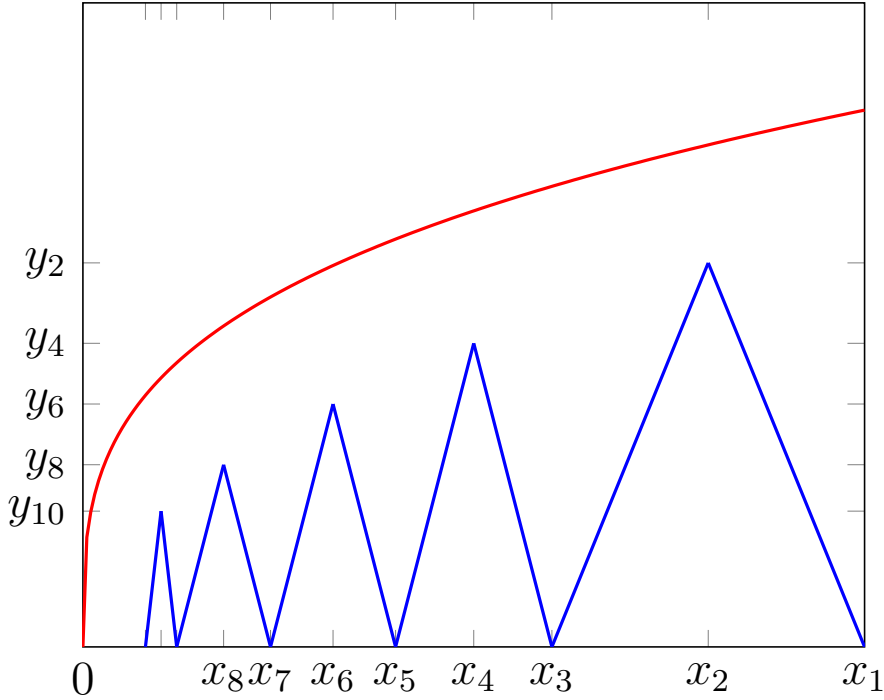
\begin{center}
\begin{tikzpicture}[scale=1.5]
\begin{axis}[xmin=0, xmax=1, ymin=0, ymax=1.5, samples=200, xtick={0,0.08,0.1,0.12,0.18,0.24,0.32,0.4,0.5,0.6,0.8,1}, xticklabels={0,,,,$x_8$,$x_7$,$x_6$,$x_5$,$x_4$,$x_3$,$x_2$,$x_1$}, ytick={0.3162,0.4243,0.5657,0.7071,0.8944}, yticklabels={$y_{10}$,$y_8$,$y_6$,$y_4$,$y_2$}]
\addplot[red, thick, domain=0:1](x,1.25*x^0.3);
\addplot[blue, thick, domain=0.8:1](x,4.4721-4.4721*x);
\addplot[blue, thick, domain=0.6:0.8](x,4.4721*x-4.4721*0.6);
\addplot[blue, thick, domain=0.5:0.6](x,7.0711*0.6-7.0711*x);
\addplot[blue, thick, domain=0.4:0.5](x,7.0711*x-7.0711*0.4);
\addplot[blue, thick, domain=0.32:0.4](x,7.0711*0.4-7.0711*x);
\addplot[blue, thick, domain=0.24:0.32](x,7.0711*x-7.0711*0.24);
\addplot[blue, thick, domain=0.18:0.24](x,7.0711*0.24-7.0711*x);
\addplot[blue, thick, domain=0.12:0.18](x,7.0711*x-7.0711*0.12);
\addplot[blue, thick, domain=0.1:0.12](x,15.8114*0.12-15.8114*x);
\addplot[blue, thick, domain=0.08:0.1](x,15.8114*x-15.8114*0.08);
\end{axis}
\end{tikzpicture}
\end{center}

The blue graph is the function $f$ on the interval $[x_{11},x_1]$, the red graph is the function $x\mapsto x^\alpha$. The values $y_{2n}$ were chosen smaller than $x_{2n}^\alpha$ in order to ensure that $f$ is $\alpha$-Hölder continuous at $0$. It remains to choose the sequences $(x_n)$ and $(y_n)$ appropriately.

First, the variation function $\var_f$ is easy to determine. Using Lemma \ref{lem:TotalVariationSigmaSummable}, we have $$\var_f(x_{2n-1}) = \var\big(f;0,x_{2n-1}\big)= 2\sum_{k=n}^\infty y_{2k}.$$

We want that $\var_f$ is $\gamma$-Hölder continuous for no $\gamma\in(0,1)$. In order to achieve this, we can choose the sequence $(y_n)$ to be decreasing as slowly as possible. Since $f$ should be of bounded variation, however, it needs to fall faster than $n^{-1}$, as otherwise, the series diverges. Therefore, we set $$y_{2n}=\frac{1}{2n\big(\log(n+1)\big)^2}.$$ With this choice $f$ is of bounded variation since $$\var(f) = \var_f(x_1) = \sum_{n=1}^\infty\frac{1}{n\big(\log(n+1)\big)^2} <\infty.$$

Now we have to choose the sequence $(x_n)$. Its decay should be slow enough so that $f$ is $\alpha$-Hölder continuous, but fast enough so that $\var_f$ is $\gamma$-Hölder continuous for no $\gamma\in(0,1)$. We set $$x_{2n-1}=n^{-\beta}$$ for an appropriate choice of $\beta>0$ that remains to be determined, and $$x_{2n}=\frac{x_{2n-1}+x_{2n+1}}{2}.$$

First, note that $$\var_f(n^{-\beta}) = \var_f(x_{2n-1})=\sum_{k=n}^\infty\frac{1}{n\big(\log(n+1)\big)^2} \ge \int_{n+1}^\infty\frac{1}{x(\log x)^2}\dint x = \frac{1}{\log (n+1)}.$$ Therefore, for $\gamma\in(0,1)$, we have $$\sup_{x\in(0,x_1]}\frac{\var_f(x)}{x^\gamma}\ge\sup_{n\in\N}\frac{\var_f(n^{-\beta})}{n^{-\beta\gamma}}\ge\sup_{n\in\N}\frac{n^{\beta\gamma}}{\log(n+1)}=\infty,$$ since $\beta\gamma>0$. Hence, $\var_f$ is not $\gamma$-Hölder continuous regardless of our choice of $\beta>0$.

It remains to ensure that $f$ is $\alpha$-Hölder continuous. First, $f$ needs to be $\alpha$-Hölder continuous at $0$, i.e. 
\begin{align*}
\sup_{x\in(0,x_1]}\frac{f(x)}{x^\alpha} &\le \sup_{n\in\N}\frac{f(x_{2n})}{x_{2n+3}^\alpha}\le\sup_{n\in\N}\frac{\frac{1}{2n(\log(n+1))^2}}{(n+2)^{-\alpha\beta}} = \sup_{n\in\N}\frac{(n+2)^{\alpha\beta}}{2n(\log(n+1))^2} \le \sup_{n\in\N}\frac{(3n)^{\alpha\beta}}{2n(\log 2)^2}\\
&\le \frac{3^{\alpha\beta}}{2(\log 2)^2}\sup_{n\in\N}n^{\alpha\beta-1}<\infty.
\end{align*}
Therefore, we choose $\beta$ such that $0<\beta\le \alpha^{-1}$.

Second, due to the specific structure of $f$, it is apparent that
\begin{align*}
\sup_{x,y\in(0,x_1]}\frac{\big|f(x)-f(y)\big|}{|x-y|^\alpha} &= \sup_{n\in\N}\frac{f(x_{2n})}{\big(\frac{x_{2n-1}-x_{2n+1}}{2}\big)^\alpha} = \sup_{n\in\N}\frac{2^\alpha\frac{1}{2n(\log(n+1))^2}}{\big(n^{-\beta}-(n+1)^{-\beta}\big)^\alpha}\\
&\le \sup_{n\in\N}\frac{\frac{2^{\alpha-1}}{n(\log 2)^2}}{((n+1)^{-\beta-1})^\alpha} = \frac{2^{\alpha-1}}{(\log 2)^2}\sup_{n\in\N}\frac{(n+1)^{\alpha(\beta+1)}}{n}\\
&\le \frac{2^\alpha}{(\log 2)^2}\sup_{n\in\N}\frac{(n+1)^{\alpha(\beta+1)}}{n+1}\\
&\le \frac{2^\alpha}{(\log 2)^2}\sup_{n\in\N}n^{\alpha(\beta+1)-1}.
\end{align*}
The last supremum is finite if $\alpha(\beta+1)-1\le 0$, i.e. if $\beta\le \alpha^{-1}-1$. Hence, $f$ is $\alpha$-Hölder continuous if $$0<\beta\le\min\Big\{\alpha^{-1},\alpha^{-1}-1\Big\} = \alpha^{-1}-1.$$ Since $\alpha<1$, the choice of such a $\beta>0$ is possible. Therefore, the function $f$ constructed this way is $\alpha$-Hölder continuous, but $\var_f$ is $\gamma$-Hölder continuous for no $\gamma\in(0,1)$.

% % % % % % % % % % % % % % % % % % % % % %
\section{Proof of Proposition \ref{prop:VariationFunctionParentFunctionModuliOfContinuity} and Theorem \ref{thm:TheTheorem}} \label{sec3}
% % % % % % % % % % % % % % % % % % % % % % 

First, we prove Proposition \ref{prop:VariationFunctionParentFunctionModuliOfContinuity}.

\begin{proof}[Proof of Proposition \ref{prop:VariationFunctionParentFunctionModuliOfContinuity}]
Since $f$ is continuous and of bounded variation, also $\var_f$ is continuous by Proposition \ref{prop:VariationFunctionParentFunctionConnections}. Therefore, $\omega_{\var_f}$ is well-defined. Now, for $a\le x\le y\le b$ with $y-x\le h$ we have with Lemma \ref{lem:VariationIsSummable} that $$\big|f(y)-f(x)\big|\le\var(f;x,y) = \var_f(y)-\var_f(x)\le \omega_{\var_f}(y-x)\le\omega_{\var_f}(h).$$ Taking the supremum over all $x,y$ as above yields $\omega_f(h)\le\omega_{\var_f}(h)$.
\end{proof}

The proof of Theorem \ref{thm:TheTheorem} is more involved than the example in Section \ref{sec4}. First, we show the following lemma.

\begin{lem}\label{lem:MinimalModulusOfContinuityProperties}
Let $f:[a,b]\to\R$ be a continuous function. Then $\omega_f$ is a modulus of continuity for $f$, is subadditive and satisfies $\omega_{\omega_f}=\omega_f$. Moreover, if $\omega$ is a modulus of continuity for $f$, then $\omega_f\le\omega$.
\end{lem}

\begin{proof}
It is obvious from the definition that $\omega_f(0)=0$ and that $\omega_f$ is increasing. Furthermore, note that $\omega_f(h)$ is finite for all $h\in[0,\infty)$. This is because $f$ is continuous on the compact set $[a,b]$, and hence bounded. We show that $\omega_f$ is subadditive. Let $s,t\ge 0$. Then 
\begin{align*}
\omega_f(s+t) &= \sup\Big\{\big|f(x)-f(y)\big|\colon x,y\in[a,b],\ |x-y|\le s+t\Big\}\\
&= \sup\Big\{\big|f(x)-f(z)+f(z)-f(y)\big|\colon x,y,z\in[a,b],\ |x-z|\le s,|z-y|\le t\Big\}\\
&\le \sup\Big\{\big|f(x)-f(z)\big|+\big|f(z)-f(y)\big|\colon x,y,z\in[a,b],\ |x-z|\le s,|z-y|\le t\Big\}\\
&\le \sup\Big\{\big|f(x)-f(z)\big|\colon x,z\in[a,b],\ |x-z|\le s\Big\}\\
&\quad + \sup\Big\{\big|f(z)-f(y)\big|\colon y,z\in[a,b],\ |y-z|\le t\Big\}\\
&= \omega_f(s)+\omega_f(t).
\end{align*}

Next, we show that $\omega_f$ is continuous at zero. Since $f$ is continuous on the compact set $[a,b]$, it is uniformly continuous. Hence, for all $\eps>0$ there exists a $\delta>0$ such that $\big|f(x)-f(y)\big|\le \eps$ for all $|x-y|\le\delta$ with $x,y\in[a,b]$. In particular, $\omega_f(\delta)\le\eps$. Since $\eps$ was arbitrary and $\omega_f$ is increasing, we have $\omega_f(0+)=\omega_f(0)=0$.

Next, we prove that $\omega_f$ is continuous everywhere. Let $t,h>0$. Since $\omega_f$ is subadditive and increasing, $$\omega_f(t)\le\omega_f(t+h)\le \omega_f(t) + \omega_f(h).$$ Taking $h$ to zero and using that $\omega_f(0+)=0$ yields that $\omega_f$ is right-continuous. The left-continuity of $\omega_f$ follows similarly from $$\omega_f(t)\le\omega_f(t-h)+\omega_f(h)\le\omega_f(t) + \omega_f(h).$$ Altogether, $\omega_f$ is continuous.

We have shown that $\omega_f$ is a modulus of continuity. Now it is trivial that $\omega_f$ is also a modulus of continuity for $f$. To show that $\omega_{\omega_f}=\omega_f$, let $h\ge 0.$ Since $\omega_f$ is increasing, 
\begin{align*}
\omega_{\omega_f}(h) &= \sup\Big\{\big|\omega_f(x)-\omega_f(y)\big|\colon x,y\ge 0,\ |x-y|\le h\Big\}\\
&= \sup\Big\{\omega_f(x+h)-\omega_f(x)\colon x\ge 0\Big\} \ge \omega_f(0+h)-\omega_f(0) = \omega_f(h).
\end{align*}
On the other hand, since $\omega_f$ is subadditive, $$\omega_{\omega_f}(h) = \sup\Big\{\omega_f(x+h)-\omega_f(x)\colon x\ge 0\Big\} \le \sup\Big\{\omega_f(x) + \omega_f(h)-\omega_f(x)\colon x\ge 0\Big\} = \omega_f(h).$$

Finally, let $\omega$ be another modulus of continuity for $f$. If there exists an $h\ge 0$ with $\omega(h)< \omega_f(h)$, then there are two points $x,y\in[a,b]$ with $|x-y|\le h$ and $\big|f(x)-f(y)\big| > \omega(h)$. Since $\omega$ is a modulus of continuity for $f$, and since $\omega$ is increasing, $$\omega(h) < \big|f(x)-f(y)\big| \le \omega\big(|x-y|\big) \le \omega(h),$$ a contradiction.
\end{proof}

We require a modulus of continuity to be increasing and continuous. However, we need additional regularity properties. The following lemmas show that we can assume those regularity properties without loss of generality.

\begin{lem}\label{lem:BoundedModuliOfContinuityHaveEventuallyConstantMajorant}
Let $\omega$ be a bounded modulus of continuity. Then there exists a modulus of continuity $\omega'\ge\omega$ with $\omega'(h)=\omega'(1)$ for all $h\ge 1$.
\end{lem}

\begin{proof}
Clearly, the function
\[\omega'(h)=
\begin{cases}
\omega(h) + \big(\|\omega\|_\infty-\omega(1)\big)h & h\in[0,1]\\
\|\omega\|_\infty & h\in(1,\infty).
\end{cases}
\]
is a modulus of continuity, $\omega' \ge \omega$, and $\omega'(h)=\omega'(1)$ for $h\ge 1$.
\end{proof}

\begin{lem}\label{lem:ModuliOfContinuityHaveReproducingMajorant}
Let $\omega$ be a modulus of continuity with $\omega(h)=\omega(1)$ for $h\ge 1$. Then $\omega_\omega \ge \omega$, and $\omega_\omega(h) = \omega_\omega(1)$ for $h\ge 1$.
\end{lem}

\begin{proof}
First, for $h\ge 0$ we have $$\omega_\omega(h) = \sup\Big\{\big|\omega(x)-\omega(y)\big|\colon x,y\ge 0,\ |x-y|\le h\Big\} \ge \big|\omega(h)-\omega(0)\big| = \omega(h).$$

Second, notice that $0=\omega(0)\le\omega(h)\le \omega(1)$ for all $h\ge 0$, since $\omega$ is increasing. Hence, $$\omega_{\omega}(h) = \sup\Big\{\big|\omega(x)-\omega(y)\big|\colon x,y\ge 0,\ |x-y|\le h\Big\} \le \omega(1)-\omega(0) = \omega(1).$$ On the other hand, for $h\ge 1$, $$\omega_{\omega}(h) = \sup\Big\{\big|\omega(x)-\omega(y)\big|\colon x,y\ge 0, |x-y|\le h\Big\} \ge \big|\omega(1)-\omega(0)\big| = \omega(1).$$ Hence, $\omega_\omega(h)=\omega_\omega(1) = \omega(1)$ for $h\ge 1.$
\end{proof}

\begin{lem}\label{lem:ReproducingModuliOfContinuityHaveConcaveMajorant}
Let $\omega$ be a modulus of continuity that satisfies $\omega_\omega=\omega$ and $\omega(h)=\omega(1)$ for $h\ge 1$. Then there exists a concave modulus of continuity $\omega'$ with $\omega'\ge\omega$, $\omega_{\omega'}=\omega'$ and $\omega'(h)=\omega'(1)$ for $h\ge 1$.
\end{lem}

\begin{proof}
Define $\omega'$ as the concave majorant of $\omega$, i.e. $$\omega'(h):=\inf\Big\{\alpha h + \beta\colon \alpha t + \beta \ge \omega(t)\text{ for all }t\ge 0\Big\}.$$ Clearly, $\omega'\ge \omega$. In particular, $\omega'$ is non-negative and $\omega'(h)\ge\omega(h)=\omega(1)$ for $h\ge 1.$ Also, since $\omega(1)\ge \omega(t)$ for all $t\ge 0$, $\omega'(h)\le \omega(1)$ for all $h\ge 0$. Therefore, $\omega'(h)=\omega'(1)=\omega(1)$ for all $h\ge 1$.

We show that $\omega'(0)=0$. If $\omega(h)=0$ for all $h\ge 0$, this is trivial. Otherwise, since $\omega(0+)=\omega(0)=0$, for all $\omega(1)>\eps>0$ there exists a $\delta>0$ such that $\omega(h)<\eps$ for $h\le \delta$. Define $$\alpha = \frac{\omega(1)-\eps}{\delta}.$$ Then $\alpha t + \eps\ge\omega(t)$ for all $t\ge 0.$ Since $\eps>0$ was arbitrary, $\omega'(0)=0$.

Next, we show that $\omega'$ is increasing. Since $\omega$ is non-negative, we can restrict the infimum in the definition of $\omega'$ to non-negative values of $\alpha$ (negative values of $\alpha$ lead to negative values of $\alpha t + \beta$ for $t$ sufficiently large). Let $t,h,\eps>0$ and let $\alpha\ge 0$, $\beta\in\R$ be such that $$\omega(s)\le\alpha s + \beta \text{ for all } s\ge 0$$ and $$\omega'(t+h) \ge \alpha (t+h) + \beta - \eps.$$ Then, $$\omega'(t) \le \alpha t + \beta \le \alpha (t+h) + \beta \le \omega'(t+h) + \eps.$$ Since $\eps>0$ was arbitrary, we have $\omega'(t)\le\omega'(t+h)$, and $\omega'$ is increasing.

Now we show that $\omega'$ is continuous. Let $t\ge 0$. Since $\omega'$ is concave, $$\omega'\big(\lambda t + (1-\lambda)x\big) \ge \lambda \omega'(t) + (1-\lambda)\omega'(x)$$ for $\lambda\in[0,1]$. Taking $x=0$ and letting $\lambda$ tend to one, we have $$\omega'(t-)\ge \omega'(t),$$ at least if $t\neq 0$. Since $\omega'$ is increasing, $\omega'(t-)=\omega'(t)$. On the other hand, $$\omega'(t) = \omega'\bigg(\lambda (t-\lambda) + (1-\lambda)\Big(t+\frac{\lambda^2}{1-\lambda}\Big)\bigg) \ge \lambda\omega'(t-\lambda) + (1-\lambda)\omega'\bigg(t+\frac{\lambda^2}{1-\lambda}\bigg).$$ Taking $\lambda$ to zero yields $$\omega'(t) \ge \omega'(t+).$$ Again since $\omega'$ is increasing, $$\omega'(t+) = \omega'(t) = \omega'(t-).$$ In particular, $\omega'$ is continuous.

It remains to show that $\omega_{\omega'} = \omega'.$ We show that $\omega'$ is subadditive, the proof is then analogous to the proof of Lemma \ref{lem:MinimalModulusOfContinuityProperties}.  Since $\omega'$ is concave, we have $$\omega'(\lambda x) = \omega'\big(\lambda x + (1-\lambda) 0 \big)\ge \lambda\omega'(x) + (1-\lambda)\omega'(0) = \lambda\omega'(x)$$ for $x\ge 0$, $\lambda\in[0,1]$. Let $s,t\ge 0$. Then
\begin{align*}
\omega'(s+t) &= \frac{s}{s+t}\omega'(s+t) + \frac{t}{s+t}\omega'(s+t) \le \omega'\bigg(\frac{s}{s+t}(s+t)\bigg) + \omega'\bigg(\frac{t}{s+t}(s+t)\bigg)\\
&= \omega'(s) + \omega'(t).
\end{align*}
\end{proof}

We mainly exploit the following two properties of concave functions. We state them without proof, since they are well-known and elementary.

\begin{lem}\label{lem:ConcaveFunctionDecreasingDerivative}
Let $\I$ be a bounded or unbounded interval, and let $g:\I\to\R$ be a concave function. Let $x,y,x+h,y+h\in\I$ where $x\ge y $ and $h\ge 0$. Then $$g(x+h)-g(x) \le g(y+h)-g(y).$$
\end{lem}

\begin{lem}\label{lem:ConcaveImpliesLipschitz}
Let $g:[0,1]\to\R$ be a concave increasing function. Then $g$ is Lipschitz continuous on all intervals $[\eps,1]$ with $\eps>0$.
\end{lem}

We can now prove Theorem \ref{thm:TheTheorem}.

\begin{proof}[Proof of Theorem \ref{thm:TheTheorem}]
Using Lemma \ref{lem:BoundedModuliOfContinuityHaveEventuallyConstantMajorant}, Lemma \ref{lem:ModuliOfContinuityHaveReproducingMajorant} and Lemma \ref{lem:ReproducingModuliOfContinuityHaveConcaveMajorant}, we can assume without loss of generality that $\omega'(h)=\omega'(1)$ for $h\ge 1$, $\omega_{\omega'} = \omega'$, and $\omega'$ is concave.

Define the function $V:[0,1]\to\R$, $V(x)=\omega'(x)$. Then $\omega_V=\omega'$. We construct a non-negative function $f$ such that $\omega$ is a modulus of continuity for $f$ and $\var_f=V$ inductively on the intervals $[x_1,x_0], [x_2,x_1],\dots$ with $x_0=1$ and $x_n\to 0$. 

Assume we have already constructed $f$ on the interval $[x_n,1]$. If $x_n=0$, we have already defined $f$ on the entire interval $[0,1]$. Otherwise, we define $x_{n+1}$ and construct $f$ on the interval $[x_{n+1},x_n]$. First, to every point $x\in[0,x_n]$, we assign a point $y_x\in[x,x_n]$ with the property that $$V(y_x) = \frac{V(x)+V(x_n)}{2}.$$ Such a point $y_x$ exists, since $V$ is increasing and continuous. Define the set $$A_{n+1}:=\Big\{x\in[0,x_n]\colon V(x+h)-V(x)\le\omega(h)\text{ for all } h\in[0,y_x-x]\Big\}.$$ Since both $V$ and $\omega$ are continuous, the set $A_{n+1}$ is closed, and thus compact. It is non-empty since $x_n\in A_{n+1}$. Therefore, $$x_{n+1}:=\inf A_{n+1}\in A_{n+1}.$$ Furthermore, we define $$y_{n+1}:=y_{x_{n+1}}.$$ Finally, we define the function $f$ on $[x_{n+1},x_n]$ as
$$f(z)=
\begin{cases}
V(z)-V(x_{n+1}) & z\in[x_{n+1},y_{n+1}]\\
V(x_n)-V(z) & z\in[y_{n+1},x_n].
\end{cases}
$$

We note some simple facts about the function $f$. We always have $f(x_n)=0$ and $$f(y_n)=\frac{V(x_{n-1})-V(x_n)}{2}.$$ Since $V$ is continuous, $f$ is continuous where it is defined. Since $V$ is increasing, $f$ is piecewise monotone; $f$ is increasing on the intervals $[x_n,y_n]$ and decreasing on the intervals $[y_{n+1},x_n]$. Since $V$ is concave, $f$ is concave on the intervals $[x_n,y_n]$ and convex on the intervals $[y_{n+1},x_n]$.

The above picture shows what $f$ might look like, at least on the interval $[x_n,x_{n-1}]$. The red function is the variation function $V$, the blue function is the parent function $f$. On the interval $[x_n,y_n]$, $f(z)=V(z)+c$, and on the interval $[y_n,x_{n-1}]$, $f(z)=-V(z)+c'$. This construction already suggests that $\var_f=V$. The constants $c$ and $c'$ are chosen such that $f(x_n)=f(x_{n-1})=0$, and the point $y_n$ is chosen such that $f$ is continuous. The point $x_n$ is chosen such that $\omega$ is a modulus of continuity for $f$ (a priori at least on the interval $[x_n,y_n]$).

\begin{figure}
\begin{tikzpicture}[scale = 1.5]
\begin{axis}[clip mode = individual, xmin = 0, xmax = 1, ymin = 0, ymax = 1, samples = 200, xtick = {0,0.2,0.4871,0.9}, xticklabels = {0,$x_n$,$y_{x_n}=y_n$,$x_{n-1}$}, ytick = {0.2508,0.4472,0.6979,0.9487}, yticklabels = {0.25,0.45,0.70,0.95}]
\addplot[red, thick, domain=0.2:0.9](x,x^0.5);
\addplot[blue, thick, domain=0.2:0.4871](x,x^0.5-0.2^0.5);
\addplot[blue, thick, domain=0.4871:0.9](x,-x^0.5+0.9^0.5);
\node at (axis cs:0.2,0) [circle, scale=0.3, draw=black!80,fill=black!80] (a) {};
\node at (axis cs:0.4871,0.2508) [circle, scale=0.3, draw=black!80,fill=black!80] (b) {};
\node at (axis cs:0.9,0) [circle, scale=0.3, draw=black!80,fill=black!80] (c) {};
\node at (axis cs:0.2,0.4472) [circle, scale=0.3, draw=black!80,fill=black!80] (d) {};
\node at (axis cs:0.4871,0.6979) [circle, scale=0.3, draw=black!80,fill=black!80] (e) {};
\node at (axis cs:0.9,0.9487) [circle, scale=0.3, draw=black!80,fill=black!80] (f) {};
\draw[] (a) -- (d) {};
\draw[] (b) -- (e) {};
\draw[] (c) -- (f) {};
\end{axis}
\end{tikzpicture}
\end{figure}

The remaining proof is split into four steps. First, we show that $(x_n)$ converges to zero. Hence, we have defined the function $f$ on the interval $(0,1]$. Second, we prove that $f(0+)=0$, and, therefore, extend $f$ continuously to $[0,1]$ with $f(0)=0$. Then, we show that $\omega_f\le\omega$ and finally, we prove that $\var_f=V$.

1. Clearly, $(x_n)$ is decreasing and bounded from below by zero. Thus, $(x_n)$ converges, say to the point $x\in[0,1]$. Assume that $x\neq 0$. Since $V$ is concave, it is Lipschitz continuous with constant $L$ on $[x/2,1]$ by Lemma \ref{lem:ConcaveImpliesLipschitz}. Since $$\lim_{h\to 0}\frac{\omega(h)}{h} = \infty,$$ there exists an $\eps>0$ such that $\omega(h)\ge Lh$ for all $h\in[0,\eps]$. Let $n\in\N$ be sufficiently large such that $0\le x-x_n\le \eps/2$. Define $z_{n+1}:=\max\big\{x/2,x_n-\eps\big\}\in\big[x/2,1\big]$. Then $$V(z_{n+1}+h)-V(z_{n+1})\le Lh\le\omega(h)$$ for $h\in[0,\eps]$. Hence, $z_{n+1}\in A_{n+1}$ and $x_{n+1}=\min A_{n+1}\le z_{n+1}<x$, a contradiction. Therefore, $(x_n)$ converges to zero. In particular, we have also shown that $(x_n)$ is strictly decreasing.

2. If the sequence $(x_n)$ is finite, this statement is trivial, since then $x_n=0$ for some $n\in\N$. If $(x_n)$ is infinite, it suffices to show that $f(y_n)$ converges to zero. Suppose this is not the case. Then there exists an $\eps>0$ such that $f(y_n)\ge \eps$ for infinitely many $n\in\N$. Let $(y_{n_k})_k$ be a subsequence of $(y_n)$ with $f(y_{n_k})\ge\eps$. Since $V$ is increasing, $$V(1)-V(0)\ge V(y_{n_1})-V(x_{n_k}) \ge \sum_{j=1}^k\Big(V(y_{n_j})-V(x_{n_j})\Big) = \sum_{j=1}^k f(y_{n_j}) \ge k\eps$$ for all $k\in\N$, a contradiction. Hence, $f(0+)=0$ and we extend $f$ continuously to $[0,1]$ with $f(0)=0$.

3. Let $h\ge 0$. We show that $\omega_f(h)\le \omega(h)$, i.e. $$\sup\Big\{\big|f(x)-f(y)\big|\colon x,y\in[0,1],\ |x-y|\le h\Big\}\le\omega(h).$$ Since $\omega$ is increasing, it suffices to show that $$\sup\Big\{\big|f(x)-f(y)\big|\colon x,y\in[0,1],\ |x-y|= h\Big\}\le\omega(h).$$ This in turn is equivalent to $$\sup\Big\{\big|f(x+h)-f(x)\big|\colon x\in[0,1-h]\Big\}\le\omega(h).$$ Let $x\in[0,1-h]$. It remains to show that $$\big|f(x+h)-f(x)\big|\le\omega(h).$$ We also write $y$ instead of $x+h$. We distinguish several different cases depending on the positions of $x$ and $y$ relative to the points $x_n$ and $y_n$. To every point $z\in(0,1]$, we can assign $n(z)\in\N$ such that $x_{n(z)}<z\le x_{n(z)-1}$. The special case $x=0$ is treated at the very end as Case 3.

Case 1. We have $n:=n(x)=n(y)$. We distinguish whether $x,y$ are in the intervals $[x_n,y_n]$ or $[y_n,x_{n-1}]$.

Case 1.1. We have $x,y\in[x_n,y_n]$. Using Lemma \ref{lem:ConcaveFunctionDecreasingDerivative},
\begin{align*}
\big|f(y)-f(x)\big| &= \Big|V(y)-V(x_n) - \big(V(x)-V(x_n)\big)\Big| = \big|V(y)-V(x)\big|\\
&= V(y)-V(x) \le V(x_n+h)-V(x_n) \le \omega(h).
\end{align*}

Case 1.2. We have $x,y\in[y_n,x_{n-1}]$. Here, we need an additional distinction on the difference $h=y-x$.

Case 1.2.1. Assume that $h\le y_n-x_n$. Using Lemma \ref{lem:ConcaveFunctionDecreasingDerivative} and Case 1.1,
\begin{align*}
\big|f(y)-f(x)\big| &= \Big|V(x_{n-1})-V(y)-\big(V(x_{n-1})-V(x)\big)\Big| = \big|V(x)-V(y)\big|\\
&= V(x+h)-V(x) \le V(x_n+h)-V(x_n) \le \omega(h).
\end{align*}

Case 1.2.2. Assume that $h\ge y_n-x_n$. Using the defining property of $y_n$,
\begin{align*}
\big|f(y)-f(x)\big| &= \Big|V(x_{n-1})-V(y)-\big(V(x_{n-1})-V(x)\big)\Big| = \big|V(x)-V(y)\big|\\
&= V(y)-V(x) \le V(x_{n-1})-V(y_n) = V(y_n)-V(x_n)\\
&\le \omega(y_n-x_n) \le \omega(h).
\end{align*}

Case 1.3. We have $x\in[x_n,y_n]$ and $y\in[y_n,x_{n-1}]$. Using the preceding cases,
\begin{align*}
\big|f(y)-f(x)\big| &\le \max\Big\{\big|f(y)-f(y_n)\big|,\big|f(x)-f(y_n)\big|\Big\}\\
&\le \max\Big\{\omega(y-y_n),\omega(y_n-x)\Big\}\\
&= \omega\Big(\max\big\{y-y_n,y_n-x\big\}\Big)\le\omega(h).
\end{align*}

Case 2. We have $m:=n(y)<n(x)=:n$. We again distinguish several different cases and reduce them all to Case 1.

Case 2.1. We have $x\in[x_n,y_n]$.

Case 2.1.1. We have $y\in[x_m,y_m]$.

Case 2.1.1.1. We have $f(x)\le f(y)$. Then,
\begin{align*}
\big|f(y)-f(x)\big| &= f(y)-f(x) \le f(y) = f(y)-f(x_m)\\
&= \big|f(y)-f(x_m)\big| \le \omega(y-x_m)\le\omega(y-x) = \omega(h).
\end{align*}

Case 2.1.1.2. We have $f(y)\le f(x)$. Then,
\begin{align*}
\big|f(y)-f(x)\big| &= f(x)-f(y) \le f(x) = f(x)-f(x_{n-1}) = \big|f(x)-f(x_{n-1})\big|\\
&\le \omega(x_{n-1}-x) \le \omega(y-x) = \omega(h).
\end{align*}

Case 2.1.2. We have $y\in[y_m,x_{m-1}]$.

Case 2.1.2.1. We have $f(x)\le f(y)$. Then,
\begin{align*}
\big|f(y)-f(x)\big| &= f(y)-f(x) \le f(y) = f(y)-f(x_m)\\
&= \big|f(y)-f(x_m)\big| \le \omega(y-x_m) \le \omega(y-x)\le\omega(h).
\end{align*}

Case 2.1.2.2. We have $f(y)\le f(x)$. Then, 
\begin{align*}
\big|f(y)-f(x)\big| &= f(x)-f(y) \le f(x) = f(x) - f(x_{n-1}) = \big|f(x)-f(x_{n-1})\big|\\
&\le \omega(x_{n-1}-x)\le\omega(y-x) = \omega(h).
\end{align*}

Case 2.2. We have $x\in[y_n,x_{n-1}]$.

Case 2.2.1. We have $y\in[x_m,y_m]$.

Case 2.2.1.1. We have $f(x)\le f(y)$. Then,
\begin{align*}
\big|f(y)-f(x)\big| &= f(y)-f(x) \le f(y) = f(y)-f(x_m)\\
&= \big|f(y)-f(x_m)\big| \le \omega(y-x_m) \le \omega(y-x) = \omega(h).
\end{align*}

Case 2.2.1.2. We have $f(y)\le f(x)$. Then,
\begin{align*}
\big|f(y)-f(x)\big| &= f(x)-f(y) \le f(x) = f(x)-f(x_{n-1}) = \big|f(x)-f(x_{n-1})\big|\\
&\le \omega(x_{n-1}-x) \le \omega(y-x) = \omega(h).
\end{align*}

Case 2.2.2. We have $y\in[y_m,x_{m-1}]$.

Case 2.2.2.1. We have $f(x)\le f(y)$. Then,
\begin{align*}
\big|f(y)-f(x)\big| &= f(y)-f(x) \le f(y) = f(y) - f(x_m)\\
&= \big|f(y)-f(x_m)\big| \le \omega(y-x_m) \le \omega(y-x) = \omega(h).
\end{align*}

Case 2.2.2.2. We have $f(y)\le f(x)$. Then,
\begin{align*}
\big|f(y)-f(x)\big| &= f(x)-f(y) \le f(x) = f(x) - f(x_{n-1}) = \big|f(x)-f(x_{n-1})\big|\\
&\le \omega(x_{n-1}-x) \le \omega(y-x) = \omega(h).
\end{align*}

Case 3. We have $x=0$. Define $n:=n(h)=n(y)$. Then, $$\big|f(y)-f(x)\big| = f(h) = f(h)-f(x_n) = \big|f(h)-f(x_n)\big|\le\omega(h-x_n)\le\omega(h).$$

4. Using Lemma \ref{lem:TotalVariationSigmaSummable} and that $f$ is continuous at zero and piecewise monotone, we have for $x\in[0,1]$ with $x_n\le x\le y_n$ that
\begin{align*}
\var_f(x) &= \var(f;0,x) = \sum_{k=n}^\infty\Big(\var\big(f;x_{k+1},y_{k+1}\big) + \var\big(f;y_{k+1},x_k\big)\Big) + \var\big(f;x_n,x\big)\\
&= \sum_{k=n}^\infty\Big(f(y_{k+1})-f(x_{k+1}) + f(y_{k+1})-f(x_k)\Big) + f(x)-f(x_n)\\
&= 2\sum_{k=n}^\infty f(y_{k+1}) + f(x) = 2\sum_{k=n}^\infty\frac{V(x_k)-V(x_{k+1})}{2} + V(x)-V(x_n)\\
&= -\lim_{k\to\infty}V(x_{k+1})+V(x_n)+V(x)-V(x_n)= V(x).
\end{align*}
Similarly, for $y_{n+1}\le x\le x_n$, we have
\begin{align*}
\var_f(x) &= \var(f;0,x)\\
&= \sum_{k=n+1}^\infty\Big(\var\big(f;x_{k+1},y_{k+1}\big) + \var\big(f;y_{k+1},x_k\big)\Big)\\
&\qquad + \var\big(f;x_{n+1},y_{n+1}\big) + \var\big(f;y_{n+1},x\big)\\
&= \sum_{k=n+1}^\infty\Big(f(y_{k+1})-f(x_{k+1}) + f(y_{k+1})-f(x_k)\Big)\\
&\qquad + f(y_{n+1}) - f(x_{n+1}) + f(y_{n+1}) - f(x)\\
&= 2\sum_{k=n}^\infty f(y_{k+1}) - f(x) = 2\sum_{k=n}^\infty\frac{V(x_k)-V(x_{k+1})}{2} - \big(V(x_n)-V(x)\big)\\
&= -\lim_{k\to\infty}V(x_{k+1})+V(x_n)-V(x_n)+V(x)= V(x).
\end{align*}
\end{proof}

\begin{acknowledgement}
This paper developed as part of the author's Master's thesis at the Johannes Kepler University in Linz, Austria. The author would like to thank his supervisor Aicke Hinrichs for proposing the problem of this paper, for his support and encouragement and his suggestions for improving the paper.
\end{acknowledgement}

% Literatur
%\newpage
%\addcontentsline{toc}{section}{Literature}
%\printbibliography

%\printbibliography


\begin{thebibliography}{1}

\bibitem{A2014}
J.~Appell, J.~Bana\'{s}, and N.~Merentes.
\newblock {\em Bounded variation and around}, volume~17 of {\em De Gruyter
  Series in Nonlinear Analysis and Applications}.
\newblock De Gruyter, Berlin, 2014.

\bibitem{H1976}
F.~N. Huggins.
\newblock Some interesting properties of the variation function.
\newblock {\em Amer. Math. Monthly}, 83(7):538--546, 1976.

\bibitem{J1881}
C.~Jordan.
\newblock Sur la série de fourier.
\newblock {\em C. R. Acad. Sci. Paris 2}, pages 228--230, 1881.

\bibitem{R1979}
A.~M. Russell.
\newblock Further comments on the variation function.
\newblock {\em Amer. Math. Monthly}, 86(6):480--482, 1979.

\end{thebibliography}
\end{document}